\documentclass[11pt,english]{amsart}
\usepackage[utf8]{inputenc}

\usepackage{amsmath}
\usepackage{amsfonts}
\usepackage{mathrsfs}
\usepackage{amsthm}
\usepackage{url}
\usepackage{amssymb}
\usepackage{xspace}
\usepackage{hyperref}
\usepackage{ifthen}
\usepackage{stackengine}
\usepackage{xcolor}
\usepackage{tabularx}
\usepackage[top=1.15in,bottom=1.15in,left=1.15in,right=1.15in,marginpar=1in]{geometry}
\usepackage{caption}
\usepackage{subcaption}
\usepackage{verbatim}
\usepackage{tikz-cd}
\usepackage{tikz}
\usepackage{comment}
\stackMath

\newcommand{\p}[1]{\bigskip\noindent\emph{#1.}}

\newcommand{\nt}{\newtheorem}
\nt{theorem}{Theorem}[section]

\nt{proposition}[theorem]{Proposition}
\nt{corollary}[theorem]{Corollary}
\nt{lemma}[theorem]{Lemma}

\newcommand{\C}{\mathbb{C}}

\newcommand{\Z}{\mathbb{Z}}

\newcommand{\squig}{{\scriptstyle\sim\mkern-3.9mu}}

\newcommand{\rsquigend}{{\scriptstyle\rule{.1ex}{0ex}\rhd}}
\newcounter{sqindex}
\newcommand\squigs[1]{%
  \setcounter{sqindex}{0}%
  \whiledo {\value{sqindex}< #1}{\addtocounter{sqindex}{1}\squig}%
}
\newcommand\rsquigarrow[2]{%
  \mathbin{\stackon[2pt]{\squigs{#2}\rsquigend}{\scriptscriptstyle\text{#1\,}}}%
}

\DeclareMathOperator{\LMod}{LMod}
\DeclareMathOperator{\PMod}{PMod}

\title{A solution to the degree-$d$ Twisted Rabbit Problem}
\date {}

\author[Malavika Mukundan]{Malavika Mukundan}
\author{Rebecca R. Winarski}
\address{Malavika Mukundan \\  Department of Mathematics\\ 530 Church Street\\
University of Michigan
Ann Arbor, MI, 48109}
\email{malavim@umich.edu}
\address{Rebecca R. Winarski \\  Department of Mathematics and Computer Science\\ College of the Holy Cross\\
1 College Street
Worcester, MA 01610}
\email{rwinarsk@holycross.edu}

\begin{document}
\begin{abstract} We solve generalizations of Hubbard's twisted rabbit problem for analogues of the rabbit polynomial of degree $d\geq 2$.  The twisted rabbit problem asks: when a certain quadratic polynomial, called the Douady Rabbit polynomial, is twisted by a cyclic subgroup of a mapping class group, to which polynomial is the resulting map equivalent (as a function of the power of the generator)?  The solution to the original quadratic twisted rabbit problem, given by Bartholdi--Nekrashevych \cite{BN}, depended on the 4-adic expansion of the power of the mapping class by which we twist.  In this paper, we provide a solution that depends on the $d^2$-adic expansion of the power of the mapping class element by which we twist.
\end{abstract}
\maketitle
\section{Introduction}
Let $f:\C\rightarrow \C$ be an orientation-preserving branched cover.  Let $C_f\subset \C$ be the set of points for which $f$ is locally non-injective (the set of critical points).  The post-critical set $P_f=\{f^n(c)\mid c\in C_f, n\geq 1\}$ is the forward orbit of $C_f$.  If $P_f$ is finite, $f$ is said to be {\it post-critically finite}.  Let $g:\C\rightarrow\C$ be another post-critically finite branched cover with post-critical sets $P_g$.  We say that $f,g$ are {\it equivalent} or {\it combinatorially equivalent} (or {\it Thurston equivalent}) if there exist orientation preserving homeomorphisms $h_0,h_1:(\C, P_f)\rightarrow(\C,P_g)$ such that $h_0f=gh_1$ and $h_0$ and $h_1$ are homotopic relative to $P_f$. Thurston proved that a post-critically finite branched cover $\C\rightarrow \C$ is either equivalent to a polynomial or has a certain kind of topological obstruction \cite{DH}.  Over the past decades, much work has been directed towards determining a holomorphic map to which an unobstructed branched cover $\C\to\C$ (or $S^2\to S^2$) is equivalent  \cite{BD,BN,BLMW,BFH,BBY,spider,KL,nekra09,nek_cactus,RSY,ST,thurston_positive}.

In the 1980s, Hubbard posed the \textit{twisted rabbit problem}, which presented the challenge of classifying certain branched covers by the polynomial to which they were equivalent.  The rabbit polynomial is the quadratic polynomial $R(z)=z^2+c$ where $c\approx-0.122561+0.744862i$ for which the critical point 0 is 3-periodic.  The post-critical set consists of three points: $\{0,R(0),R^2(0)\}$.  Let $x$ be a curve surrounding $R(0)$ and $R^2(0)$, and let $D_x$ be the Dehn twist about $x$ (see Figure~\ref{fig:rabbit_with_loops}).  The composition $D_x^mR$ is a branched cover $\C\rightarrow \C$, and by the Bernstein--Levy theorem \cite{LevyThesis} (also \cite[Chapter~10]{hubbard}), it is equivalent to a polynomial.  Hubbard's twisted rabbit problem is: for $m \in \Z$, find a function in terms of $m$ that determines to which polynomial $D_x^mR$ is equivalent.  After remaining open for nearly 25 years, Bartholdi--Nekrachevych solved the twisted rabbit problem in 2006 \cite{BN}.

Belk--Lanier--Margalit and the second author solved a generalization of the twisted rabbit problem in which they compose quadratic polynomials where the critical point is $n$-periodic with powers of an analogous Dehn twist for $n\geq 3$ \cite{BLMW}.  Recently Lanier and the second author extended their work to unicritical cubic polynomials (with any number of post-critical points).

\begin{figure}
    \centering
    \includegraphics[scale=0.25]{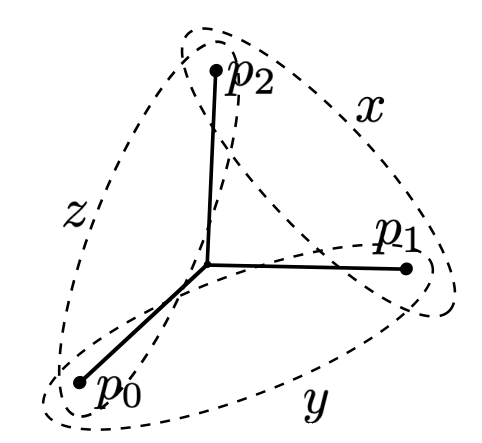}
    \caption{The Hubbard tree of the rabbit polynomial along with the simple closed curves $x$, $y$ and $z$}
    \label{fig:rabbit_with_loops}
\end{figure}
\p{The degree-$d$ twisted rabbit problem} In this paper, we generalize Hubbard's twisted rabbit problem to higher degree analogues of the rabbit polynomial.  That is: for each $d\geq 2$, there is a unicritical polynomial $R_d(z)=z^d+c_d$ that naturally generalizes the rabbit polynomial.  The critical point (0) is 3-periodic, as in the quadratic case.  Let $x_d$ be the curve that is homotopic to the boundary of a neighborhood of the straight line segment between $R_d(0)$ and $R_d^2(0)$.  In Theorem \ref{thm:degree_d}, we describe a function that determines to which polynomial $D_{x_d}^mR_d$ is equivalent in terms of $m$.

\p{The polynomials} Before we state the solution to the degree-$d$ twisted rabbit problem, we explain the notation we use for the polynomials that appear.  For each degree $d$, there exist $d+1$ equivalence classes of polynomials that have a critical point that is 3-periodic.  
One of these is $R_d$, the generalization of the rabbit polynomial (the degree-$d$ rabbit); see Figure \ref{fig:R} for the Julia set when $d=5$.  The complex conjugate of the rabbit polynomial is called the degree-$d$ corabbit polynomial $\overline{R}_d$, see Figure \ref{fig:CR} for its Julia set when $d=5$.  The remaining $d-1$ polynomials are all generalizations of the airplane polynomial.  These remaining $d-1$ polynomials have a Hubbard tree with two edges that meet at the critical point, and can be indexed by the angle made by these two edges at the critical point with a fixed orientation. For $1 \leq i \leq d-1$, we denote by $A_{d,i}$ the degree-$d$ generalization of the airplane for which this angle is $\frac{2\pi i}{d}$; see Figure \ref{fig:A1}-\ref{fig:A4} for the Julia sets of the four degree-$d$ airplanes when $d=5$.

The solution to the original (quadratic) twisted rabbit problem depends on the 4-adic expansion of the power by which we twist.  Similarly, the solution to the degree-$d$ twisted rabbit problem will depend on the $d^2$-adic expansion of the power by which we twist.

Any integer $m$ has a $d^2$-adic expansion of the form $m_sm_{s-1}m_{s-2}.....m_1$ if $m\geq 0$, or $\overline{d^2-1}m_sm_{s-1}m_{s-2}...m_1 $ if $m<0$, with $m_i \in \{0,1,...,d^2-1\}$ for all $i \in \{1,2,...,s)\}$.  Let $\sigma_{d^2}(m)$ be the least value of $s$ such that all digits of the $d^2$-adic expansion of $m$ to the left of $m_s$ are repeating (that is, for all $t>s$, $m_t=0$ if $m\geq 0$ or $m_t=d^2-1$ if $m<0$).
\begin{theorem}\label{thm:degree_d}
Let $R_d$ be the degree-$d$ rabbit polynomial and $D_x$ the Dehn twist about the curve $x=x_d$.   \\
If $(d+1)| m_i$ for all $i \in \{1,2,...,\sigma_{d^2}(m)\}$, then 
    $$D_x^mR_d \simeq \begin{cases}
    R_d &\text{if }m \geq 0\\
    \overline{R}_d&\text{if }m<0
    \end{cases}.$$
Otherwise, let $i$ be the least index such that $m_i$ is not divisible by $d+1$. We may write $m_i$ uniquely as $d\ell+n$, where $\ell,n \in \{0,1,...,d-1\}$. Since $(d+1)\nmid m_i$, we have that $\ell\neq n$.  Then:

    $$D_x^mR_d \simeq \begin{cases}
    A_{d,n-\ell} &\text{if } n>\ell\\
    A_{d,d-(\ell-n)} &\text{if }n<\ell
    \end{cases}.$$

\end{theorem}

\medskip
In particular, in the set $\{D_x^mR_d\mid |m|\leq N\}$, the airplane polynomials $A_{d,1},\cdots, A_{d,d-1}$ occur with equal frequency.  Moreover, as $N$ tends to infinity, the probability that $D_x^mR_d$ (with $|m|\leq N$) is equivalent to the rabbit or corabbit polynomial approaches zero. 
\begin{corollary}
Fix a degree $d$ greater than 1.  For $S\geq 1$, let $$\Sigma_S=\{m\in\mathbb{Z}\mid\sigma_{d^2}(m)\leq S\}.$$ With the uniform distribution on $\Sigma_S$, the probability that for $m \in \Sigma_S$, $D_x^mR_d$ is equivalent to an airplane, that is, a map in the collection $\{A_{d,i}\mid 1\leq i\leq d-1\}$, is given by $1-\frac{1}{d^S}$. In particular, for $m \in \Sigma_S$,
\begin{itemize}
    \item For any $i \in \{1,2,...,d-1\}$, the probability that $D_x^mR_d\simeq A_{d,i}$ is $\frac{1}{d-1} - \frac{1}{(d-1)d^S}$.
    \item The probability that $D_x^mR_d \simeq R_d$ is $\frac{1}{2d^S}$
    \item The probability that $D_x^mR_d \simeq \overline{R}_d$ is $\frac{1}{2d^S}$
    \end{itemize}
\end{corollary}
\p{Example} To make Theorem~\ref{thm:degree_d} concrete, we give the polynomials to which $D_x^mR_5$ is equivalent for $0\leq m\leq 24$ in Table \ref{tab:basecasesdeg5}.  We observe that $A_{5,1},\cdots, A_{5,4}$ each appear 4 times and $R_d$ appears 5 times.  The next $m$ for which $D_x^mR_5$ is equivalent to the rabbit is $m=150$, which has 25-adic expansion $(6)(6)$.  We also note that $\overline{R}_5$ does not appear in the table because it only occurs when $m<0$.  For $-25\leq m\leq 0$, the polynomial $\overline{R}_d$ occurs when $m=-1,-7,-13,-19,-25$.
\begin{table}[h]
    \centering
    \begin{tabular}{||c|m{1.5cm}||c|m{1.5cm}||c|m{1.5cm}
    ||c|m{1.5cm}||c|m{1.5cm}||}
    \hline $m$ & $D_x^m R_5$ & $m$ & $D_x^m R_5$ & $m$ & $D_x^m R_5$ & $m$ & $D_x^m R_5$ & $m$ & $D_x^m R_5$ \\ \hline 
        0 &  $R_5$ & 5 & $A_{5,4}$ & 10 & $A_{5,3}$ & 15 &$A_{5,2}$ & 20 & $A_{5,1}$ \\
         1 & $A_{5,1}$ & 6 & $R_5$ & 11 & $A_{5,4}$ & 16 & $A_{5,3}$ & 21 & $A_{5,2}$\\
        2 & $A_{5,2}$ & 7 & $A_{5,1}$ & 12 & $R_5$ & 17 & $A_{5,4}$ & 22 & $A_{5,3}$\\
        3 & $A_{5,3}$ & 8 & $A_{5,2}$ & 13 & $A_{5,1}$ & 18 & $R_5$ & 23 & $A_{5,4}$\\
        4 & $A_{5,4}$ & 9 & $A_{5,3}$ & 14 & $A_{5,2}$ & 19 & $A_{5,1}$ & 24 & $R_5$\\
        \hline 
       \end{tabular}
    \caption{Base cases for $d=5$.}
    \label{tab:basecasesdeg5}
\end{table}

\p{Methods} We follow the strategy of Bartholdi--Nekrashevych \cite{BN}: in Section~\ref{sec:reduction}, we find formulae that reduce $D_x^mR_d$ to $R_d$ post-composed with one of a finite set of maps.  Then in Section~\ref{sec:base_cases}, we determine a polynomial equivalent to each of these ``base cases" using the lifting algorithm of Belk--Lanier--Margalit and the second author in \cite{BLMW}.

\p{Other twisted rabbit problems} As mentioned above, the (quadradic) twisted rabbit problem remained open for over two decades.  When Bartholdi--Nekrashevych solved the problem, it shifted the techniques used to study holomorphic dynamics.  But they gave more than just a solution to the problem that Hubbard originally posed: for instance, they gave an algorithm to determine the polynomial to which $gR_2$ was equivalent for any pure mapping class $g$.  Their work opened up a world of possible generalizations: in this paper, we follow their lead by increasing the {\it degree} of the polynomial by which we twist.  Belk--Lanier--Margalit and the second author generalize the (quadratic) twisted rabbit problem by increasing the size of the post-critical set.  In concurrent work of Lanier and the second author, they solve several of these problems for the cubic rabbit polynomial.  In particular, they give a closed-form solution to $D_x^mR_3$ that agrees with our solution in that case.  
They also generalize this case to give a closed-form solution to determine the equivalence class of $D_x^mR_{3,n}$ where $R_{3,n}$ is a specific cubic polynomial with $n$ post-critical points that is a natural generalization of $R_3$.  
By pairing our paper with theirs, we have a 2-parameter family of generalizations of the twisted rabbit problem: by varying both the degree $d$ and the number of post-critical points.  In this paper we prove that we can obtain any equivalence class of a degree $d$-unicritical polynomial with a 3-perioidic critical point by twisting the degree $d$ rabbit polynomial by a power of $D_x$.  Lanier and the second author show that the same is not true when the critical point has period greater than 3 (at least in the cubic case).  Further investigation into this 2-parameter family of twisted rabbit problems may reveal deeper structure or patterns within the set of unicritical polynomials with periodic critical point.

\p{Acknowledgements} This material is based upon work supported by the National Science Foundation under Grant No. DMS-1928930 while the authors participated in a program hosted by the Mathematical Sciences Research Institute in Berkeley, California, during Spring 2022.  The second author was supported by the National Science Foundation under Grant No.\ DMS-2002951. 

\section{Hubbard trees and degree-\textit{d} polynomials}
Throughout this paper, we rely heavily on the established theory of Hubbard trees \cite{DH1,DH2}.  We will use a modification of Poirier's conditions for Hubbard trees to allow us to define (topological) Hubbard trees for unobstructed branched covers $\C\to\C$ \cite{DH,poirier}.

\subsection{Hubbard trees} Let $f:\C\rightarrow\C$ be a post-critically finite branched cover with post-critical set $P_f$.  For our purposes, a tree $T$ is a finite graph with no cycles, embedded in $\C$ such that:
\begin{enumerate}
    \item $P_f$ is contained in the vertex set of $T$ and
    \item all leaves of $T$ are in $P_f$.
\end{enumerate}
The preimage $f^{-1}(T)$ does not, in general, satisfy the definition of a tree given above.  The {\it lift} of a tree $T$, denoted $\widetilde{T}$, is the hull of $f^{-1}(T)$ relative to $P_f$.  That is, any edge of $f^{-1}(T)$ that is not contained in a path between a pair of points in $P_f$ is contracted to a point. Thus $\widetilde{T}$ is a tree. We say two trees $T$ and $T'$ are {\it isomorphic} if they are they are homeomorphic as topological spaces.

Let $f$ be a polynomial of degree $d\geq 2$.  Let $P_f$ be the post-critical set for $f$.  The {\it Hubbard tree} $H_f$ for $f$ is a subset of $\C$ comprised of regulated arcs of the filled Julia set of $f$ that connect pairs of points of $P_f$.  An important feature of the Hubbard tree of $f$ is that it is invariant under $f$, that is, $f(H_f)\subseteq H_f$.  Likewise, it is also true that $f^{-1}(H_f)\subseteq H_f$ and that $H_f$ is isotopic to $\widetilde{H}_f$.

Let $g$ be a branched cover $\C\to\C$ that is equivalent to $f$.  Then we can define the (topological) Hubbard tree for $g$, denoted $H_g$, as the pullback of $H_f$ under the equivalence between $f$ and $g$.  In this case, the abstract trees for $H_f$ and $H_g$ are isomorphic.

\p{Angle assignments} The combinatorial structure of a Hubbard tree is not sufficient to distinguish the combinatorial equivalence class of a branched cover $\C\to\C$.  That is, there are inequivalent polynomials that have isomorphic Hubbard trees.  There are various different types of additional information that one can provide to distinguish post-critically finite branched covers.  Following Poirier (see \cite{poirier}), we use an invariant angle assignment, which we define as follows.  Because we think of a (Hubbard) tree as a subset of $\C$, the edges that meet at any vertex have an associated cyclic order.  An {\it angle assignment for a vertex $v$} is a set of angle measures between each pair of (not necessarily distinct) edges that have endpoint $v$ and are adjacent in the cyclic order around $v$ such that the measures of the angles sum to $2\pi$.  An {\it angle assignment for a tree $T$} is the union of angle assignments at each of the vertices of $T$.  

The tree lift $\widetilde{T}$ inherits an angle assignment from $T$, as follows.  Let $\angle$ be an angle of $T$ adjacent to the vertex $v\in T$ with measure $|\angle|$.  Let $\widetilde{v}\in f^{-1}(v)$, and let $\nu(\widetilde{v})$ denote the local degree of $\widetilde{v}$ under $f$.  There exists a preimage of the angle $\angle$ at $\widetilde{v}$, denoted $\widetilde{\angle}$.  Assign the measure of $\frac{|\angle|}{\nu(\widetilde{v})}$ to the angle $\widetilde{\angle}$.  By applying this process to each angle in $T$, we define an angle assignment on $f^{-1}(T)$.  To define an angle assignment on $\widetilde{T}$, consider any edge $e$ in $f^{-1}(T)$ that is contracted to a point of $\widetilde{T}$.  The edge $e$ is the side of two angles $\angle_1$ and $\angle_2$.  When $e$ is contracted to a point of $\widetilde{T}$, the two angles $\angle_1$ and $\angle_2$ are replaced with a new angle $\angle'$.  Assign the measure of $\angle'$ to be $|\angle_1|+|\angle_2|$. We say an angle assignment on $T$ is {\it invariant} if:
\begin{enumerate}
    \item $T$ is isotopic to $\widetilde{T}$ and
    \item each angle of $T$ has the same measure as the corresponding angle (under the isotopy) of $\widetilde{T}$.
\end{enumerate}
Given a branched cover $f:\C\to\C$, it follows from the work of Douady--Hubbard \cite{DH1,DH2} or Poirier \cite{poirier} that a Hubbard tree $H_f$, an invariant angle assignment on $H_f$, and the restriction of $f$ to $\widetilde{H}_f$ suffice to determine the equivalence class of $f$.

\subsection{The degree-\textit{d} polynomials with a 3-periodic critical cycle}\label{sec:3perpolys}
In this section we describe the Hubbard trees for all unicritical polynomials of degree $d$ with a $3$-periodic critical point.  We first count the equivalence classes of such polynomials.

\p{Counting parameter rays} Every unicritical polynomial of degree $d$ is affine conjugate to $z^d+c$ for some $c \in \C$.
There exist exactly $d^2-1$ polynomials in this family that have a critical cycle of period $3$. Indeed, this may be seen by counting angles of parameter rays in the Mandelbrot set.  The parameter rays that land on a hyperbolic component of the Mandelbrot set that contains a polynomial with a critical cycle of period dividing 3 exactly comprise the set of parameter rays with angles in $\left\{\frac{i}{d^3-1}\mid i \in \{0,1,2,...,d^3-2\}\right\}$.  Among such rays, the rays at the angles $0, \frac{1}{d-1},\frac{2}{d-1},...,\frac{d-2}{d-1}$ land on the unique hyperbolic component of period 1.  Of the remaining $d^3-d $ rays, groups of $d$ rays each land on the same hyperbolic component.  Thus there are $(d^3-d)/d=d^2-1$ hyperbolic components that contain unicritical polynomials with critical cycle of period $3$. But the affine conjugacy class (and therefore the equivalence class) of each polynomial $z^d+c$ consists of the $d-1$ polynomials $z^d + \omega^j c$, where $\omega = e^{\frac{2 \pi i}{d-1}}$ and $j \in \{1,2,...,d-1\}$. Thus, there are exactly $d+1$ equivalence classes of unicritical polynomials that have a critical cycle of exact period $3$.

Alternatively, we observe that every polynomial of the form $z^d+c$ with $c\neq 0$ is conjugate to the polynomial $p_\lambda(z) = \lambda(1+\frac{z}{d})^d$, where $\lambda = dc^{d-1}$. In the polynomial family $\{p_\lambda\}$, the polynomials $p_{\lambda_1}$ and $p_{\lambda_2}$ are conjugate if and only if $\lambda_1 = \lambda_2$. In this family, there are exactly $d+1$ distinct solutions to the equation $p_\lambda^{\circ 3}(0) = 0.$

\p{Hubbard trees of polynomials with three critical points} Let $f$ be a unicritical polynomial of degree $d$ with a 3-periodic critical point.  Let $p_0$ be the critical point of $f$, let $p_1=f(p_0)$ be the critical value, and let $p_2=f^2(p_0)$.

The $d+1$ equivalence classes of degree $d$ polynomials that have a critical point of period 3 can be distinguished by their Hubbard trees and an invariant angle assignment.  There are only two combinatorial structures for (Hubbard) trees with three post-critical points: \begin{enumerate}
    \item a tripod, that is, a graph where $\{p_0,p_1,p_2\}$ are all leaves and there is an unmarked trivalent vertex or 
    \item a path of length 2 (two of the post-critical points are leaves).
\end{enumerate}
\p{The rabbit $R_d$ and corabbit $\overline{R}_d$} The rabbit polynomial $R_d$ and corabbit polynomial $\overline{R}_d$ are the two polynomials (up to equivalence) that have a tripod as their Hubbard tree.  Both polynomials cyclically permute edges.  Therefore the only possible invariant angle structure is $\frac{2\pi}{3}$ in both case.  They can be distinguished, however because $R_d$ rotates the edges of its Hubbard tree counterclockwise and $\overline{R}_d$ rotates the edges of its Hubbard tree clockwise.  Because a polynomial is determined by its Hubbard tree, the invariant angle structure, and the edge map of its Hubbard tree, we may take this to be the definition of $R_d$ and $\overline{R}_d$.  Where $d=5$, the Hubbard tree for $R_5$ is shown as a subset of the filled Julia set in Figure \ref{fig:R} and the Hubbard tree for $\overline{R}_5$ is a subset of the filled Julia set in Figure \ref{fig:CR}.

\p{The airplanes $A_{d,i}$} The remaining $d-1$ equivalence classes of polynomials have Hubbard trees that are paths of length 2.  Moreover, the vertex of valence 2 must be $p_0$ because the critical point is the only point for which $f$ maps a neighborhood $d$-to-1 to its image.  
The polynomials are distinguished by the invariant angle structure on their Hubbard tree.  Let $\angle$ be the counterclockwise angle between the edge $e_1$ with endpoints $p_0$ and $p_1$ and the edge $e_2$ with endpoints $p_0$ and $p_1$.  For any $i \in \{1,2,...,d-1\}$, there is a polynomial such that $\angle$ has an invariant angle assignment of $\frac{2\pi i}{d}$.  Define $A_{d,i}$ to be the polynomial where the invariant measure of $\angle$ is $\frac{2\pi i}{d}$.  Figures \ref{fig:A1}-\ref{fig:A4} show the filled Julia sets and Hubbard trees (with angles) for $A_{5,1}$ through $A_{5,4}$, respectively.

\begin{figure} 
\centering
     \begin{subfigure}{0.27\textwidth}
         \includegraphics[width=\textwidth]{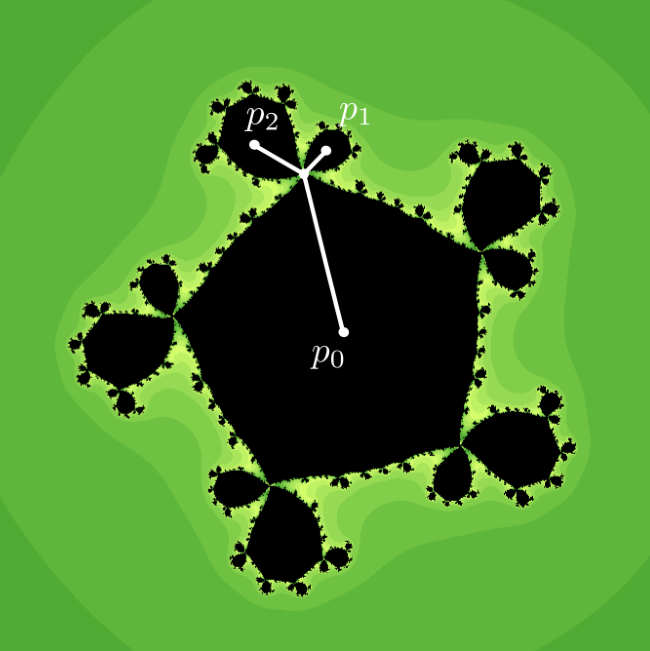}
         \caption{$R_5$}
         \label{fig:R}
     \end{subfigure}\hfill
     \begin{subfigure}{0.27\textwidth}
         \includegraphics[width=\textwidth]{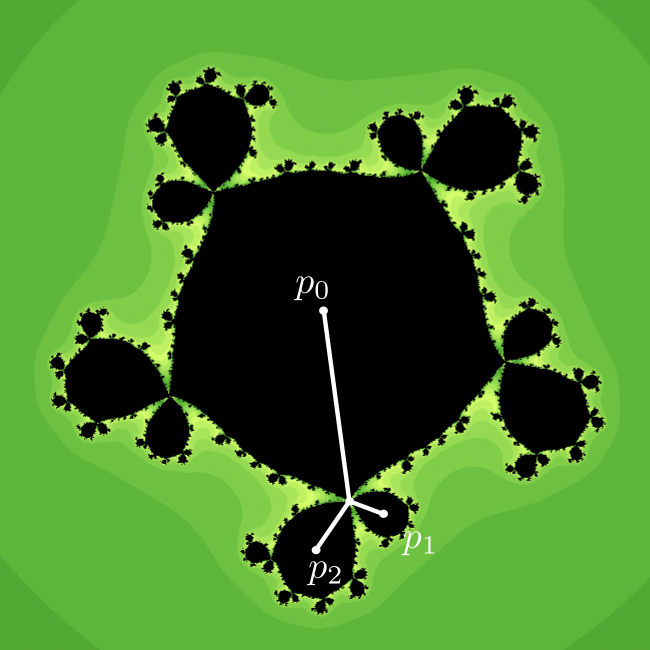}
         \caption{$\overline{R}_5$}
         \label{fig:CR}
     \end{subfigure}\hfill
     \begin{subfigure}{0.27\textwidth}
         \includegraphics[width=\textwidth]{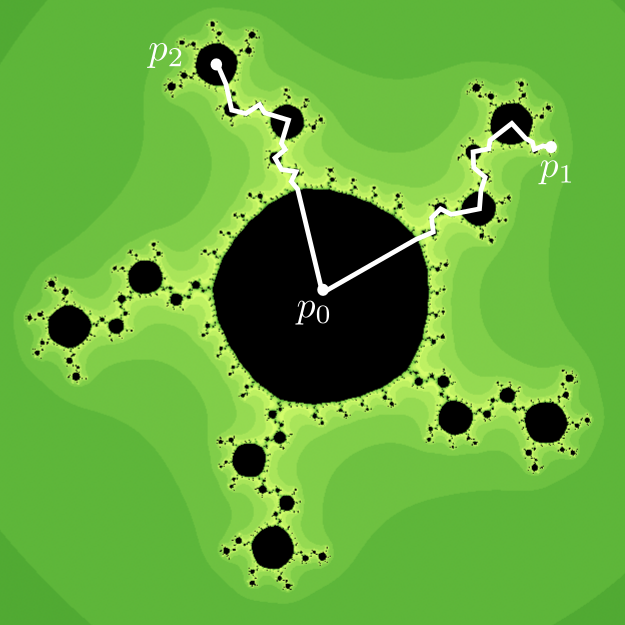}
         \caption{$A_1$}
         \label{fig:A1}
     \end{subfigure}\hfill      
     \begin{subfigure}{0.27\textwidth}
         \includegraphics[width=\textwidth]{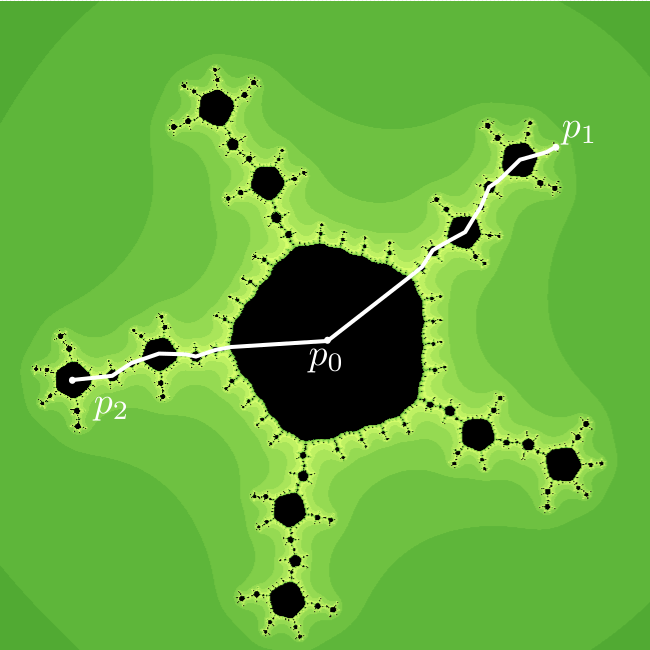}
         \caption{$A_2$}
         \label{fig:A2}
     \end{subfigure}\hfill     
     \begin{subfigure}{0.27\textwidth}
         \includegraphics[width=\textwidth]{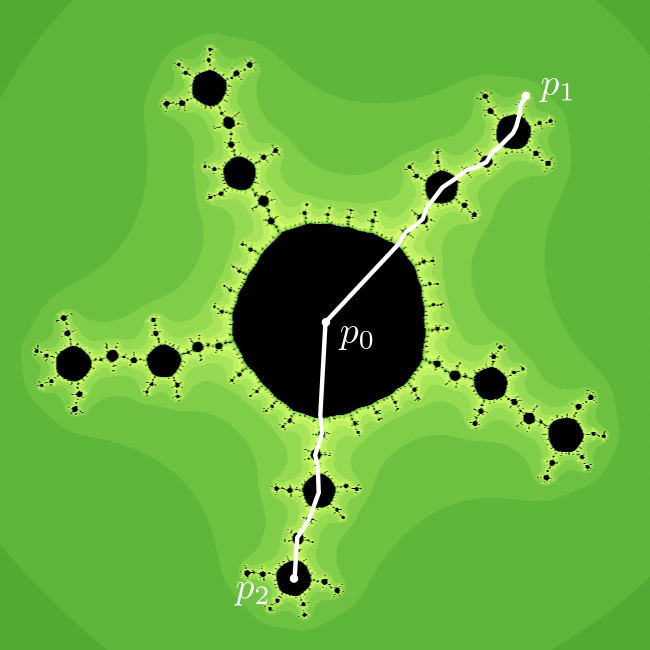}
         \caption{$A_3$}
         \label{fig:A3}
     \end{subfigure}\hfill     
     \begin{subfigure}{0.27\textwidth}
         \includegraphics[width=\textwidth]{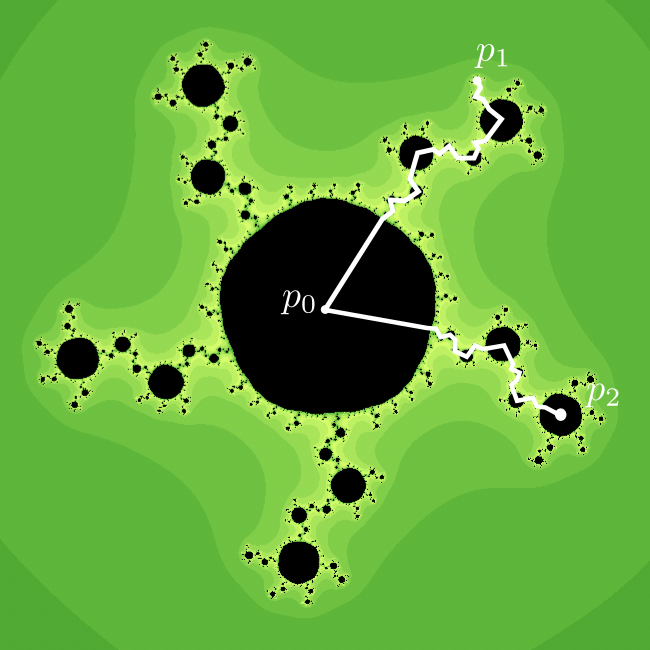}
         \caption{$A_4$}
         \label{fig:A4}
     \end{subfigure}
     \caption{The Julia sets and Hubbard trees for the unicritical polynomials of degree $5$ with 3-periodic critical point.}
     \label{fig:per3_polynomials}
\end{figure}

\section{Reduction Formulae}\label{sec:reduction}
We prove Theorem \ref{thm:degree_d} in two steps, following the original proof of Bartholdi-Nekraschevych.  In Lemma \ref{lem:reduction}, we give an algorithm that determines a map to which $D_{x_d}^mR_d$ is equivalent and that either belongs to a finite set (of base cases) or is of the form $D^k_{x_d}R_d$ where $k\leq m$.  In the next section, we find a Hubbard tree for each of the base cases.

\p{Lifting} Let $P_d$ be the post-critical set of $R_d$.  A homeomorphism $h:(\C,P_d)\rightarrow(\C,P_d)$ is said to be liftable under $R_d$ if there exists a (unique) homeomorphism $\widetilde{h}:(\C,P_d)\rightarrow(\C,P_d)$ such that $R_d\widetilde{h}=hR_d$.  In this case, $hR_d$ is equivalent to $\widetilde{h}R_d$.  We use the notation
$$h\rsquigarrow{}{3}\widetilde{h}$$ to describe the (directional) equivalence between $hR_d$ and $\widetilde{h}R_d$.

The isotopy classes of homeomorphisms that lift under $R_d$ form a finite index subgroup of $\PMod(\C,P_d)$ called the {\it liftable mapping class group} $\LMod(\C,P_d)$.  As described above, there is a homomorphism $\psi: \LMod(\C,P_d)\to\PMod(\C,P_d)$ defined for $h\in\LMod(\C,P_d)$ as $\psi(h)=\widetilde{h}$.  

Following Bartholdi--Nekrashevych, we use $\psi$ to define a version of lifting for any $h\in\PMod(\C,P_d)$.  For any $h\in\PMod(\C,P_d)$ there exists $g\in\PMod(\C,P_d)$ such that $g^{-1}h\in\LMod(\C,P_d)$.  Then $g^{-1}hR_d$ is equivalent to $\psi(g^{-1}h)R_d$.  Moreover, $hR_d$ is equivalent to $\psi(g^{-1}h)gR_d$ (\cite[Lemma 5.1]{BLMW}).  We use the notation: 
$$h\overset{g}{\rsquigarrow{}{5}}\psi(g^{-1}h)g$$ to denote the equivalence between $hR_d$ and $\psi(g^{-1}h)gR_d$ obtained by lifting $g^{-1}h$ under $R_d$.  In particular, the superscript $g$ indicates which coset representative of $h\LMod(\C,P_d)$ we choose.  If $h\in\LMod(\C,P_d)$, the notation $h\rsquigarrow{}{3}\widetilde{h}$ is a special case of the notation $$h\overset{g}{\rsquigarrow{}{5}}\psi(g^{-1}h)g$$ where $g$ is the identity, which we suppress.

\p{Branch cuts} Let $f:\C\rightarrow \C$ be a postcritically finite branched cover of degree $d\geq 2$.  A {\it branch cut} $B$ is a union of arcs such that:
\begin{itemize}
    \item each endpoint of each arc in $B$ is a critical value of $f$ (possibly infinity),
    \item each critical value of $f$ is an endpoint of an arc of $B$, and
    \item the complement of $f^{-1}(B)$ in $\C$ consists of $d$ components.
\end{itemize} If $f$ is unicritical, $B$ can be chosen to be a single arc $b$ joining the unique critical value to infinity.  A {\it special branch cut} for $f$ is a branch cut such that all points in the post-critical set $P_f$ for $f$ are contained in the closure of a single component of $\C \setminus f^{-1}(b)$.

\p{Intersections with a branch cut} Let $\gamma$ be an arc in $(\C,P_f)$ with endpoints in $P_f$, and let $b$ be a branch cut for $f$. The preimage $f^{-1}(\gamma)$ intersects the preimage $f^{-1}(b)$ at $d|\gamma\cap b|$ points.  Moreover, if we assign an orientation to $b$ and $\gamma$, the points of intersection $\gamma\cap b$ will inherit an orientation.  The orientation of each point of $\gamma\cap b$ will lift to an orientation of the corresponding $d$ preimages of $f^{-1}(\gamma)$.  The {\it geometric intersection} of $\gamma$ and $b$ is the minimum of $|\gamma\cap b|$ over all arcs homotopic to $\gamma$.  The {\it algebraic intersection} of $\gamma$ and $b$ is the sum of the signed (from the orientation) intersections of the homotopy class representative of $\gamma$ that realizes the algebraic intersection.

\p{Defining arcs} Let $c$ be a curve that bounds a disk that contains two points $p,q$ in $P_f$.  Then $c$ is homotopic to the boundary of a neighborhood of a simple arc $\gamma_c$ with endpoints at $p$ and $q$.  We call $\gamma_c$ the {\it defining arc} of $c$.  

Recall that a curve is trivial if it is homotopic to a point.  The next lemma gives a condition for when the preimage of a (non-trivial) curve is trivial.
\begin{lemma}\label{lem:branch_cut_intersection}
Let $f:\C\to\C$ be a post-critically finite branched cover of degree $d$ and let $b$ be a special branch cut for $f$.  The complement of $f^{-1}(b)$ in $\C$ consists of $d$ (open) components.  Label these components counterclockwise as $\Gamma_0,\cdots, \Gamma_{d-1}$ such $\Gamma_0$ is the component such that its closure contains the post-critical set.  

Let $c$ be a curve in $(\C,P_f)$ that bounds a disk containing two points in $P_f$ such that neither point is the critical value.  Let $\gamma_c$ be the defining arc of $c$.  Suppose the algebraic intersection of $\gamma_c$ and $b$ is $i$.  Then for each $0\leq j\leq d-1$, there is a path lift of $\gamma_c$ such that the endpoints of the path lift are in $\Gamma_j$ and $\Gamma_{(j-|i|)\mod d}$. In particular, each component of $f^{-1}(c)$ is trivial if and only if $d\nmid i$.
\end{lemma}
\begin{proof}
For the first statement, induct on $i$.

A path lift of $\gamma_c$ is non-trivial if and only if both endpoints are in $P_f\subset\Gamma_0$.  Moreover, the endpoints of a lift of $\gamma_c$ are in the same component if and only if $d\mid i$.
\end{proof}

Lemma \ref{lem:reduction} is the key step that allows us to write $D_{x_d}^mR_d$ as a map with a lower power than $m$.  Let $x=x_d$, $y=y_d$, and $z=z_d$ be the curves in Figure~\ref{fig:rabbit_with_loops}.

Let $c$ be a simple closed curve in $(\C,R_d)$.  Recall that the Dehn twist about $c$ is trivial if and only if $c$ is trivial.  For example, in  Figure~\ref{liftedcurve}, we illustrate the curve $D_y^{-2}(z)$ and its lift under $R_5$. By Lemma~\ref{lem:branch_cut_intersection}, all components of $R_5^{-1}(D_y^{-2}(z))$ are trivial since $5 \nmid -2$, therefore any lift of the $D_{D_y^{-2}(z)}$ under $R_5$ is trivial.
\begin{lemma} \label{lem:reduction}
Let $m \in \Z$, $m = d^2k+d\ell+n$ with $k \in \Z,$ and $\ell,n \in \{0,1,...,d-1\}$.  Note that $k,\ell,$ and $n$ are unique.  We have 
\begin{align*}
    D_x^mR_d \simeq \begin{cases}
     D_x^kR_d & \ell =n\\
     D_y^{\ell-n}R_d  & \ell\neq n
    \end{cases}.
\end{align*}
\end{lemma}
\begin{proof}
 We note that for any $r$, the defining arcs of the curves $D_x^{r}(z)$ and $D_y^{r}(z)$ have algebraic intersection $\pm r$ with the branch cut.  Then by Lemma~\ref{lem:branch_cut_intersection}, for all $r$ such that $d\nmid r$, the curves $D_x^{r}(z)$ and $D_y^{r}(z)$ lift to trivial curves. Thus $D_{D_x^r(z)}$ and $D_{D_y^r(z)}$ are trivial for all $r$ such that $d\nmid r$.

We consider three cases:
\begin{enumerate}
\item If $n=\ell=0$, then we have:
$$D_x^m=D_x^{d^2k} \rsquigarrow{}{3} D_y^{dk} \rsquigarrow{}{3} D_z^k \rsquigarrow{}{3} D_x^k.$$
\begin{figure} \centering
    \includegraphics[width=.25\textwidth]{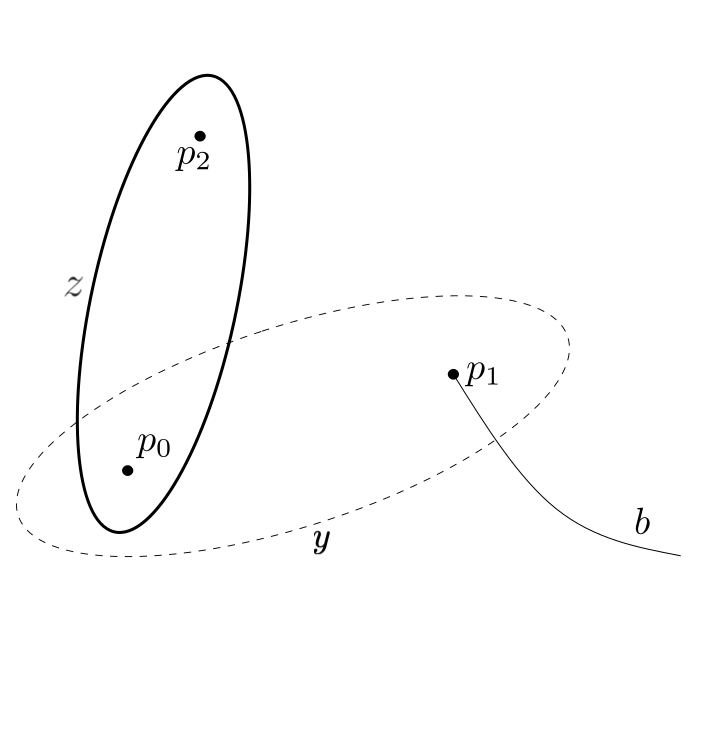}\quad
    \includegraphics[width=.25\textwidth]{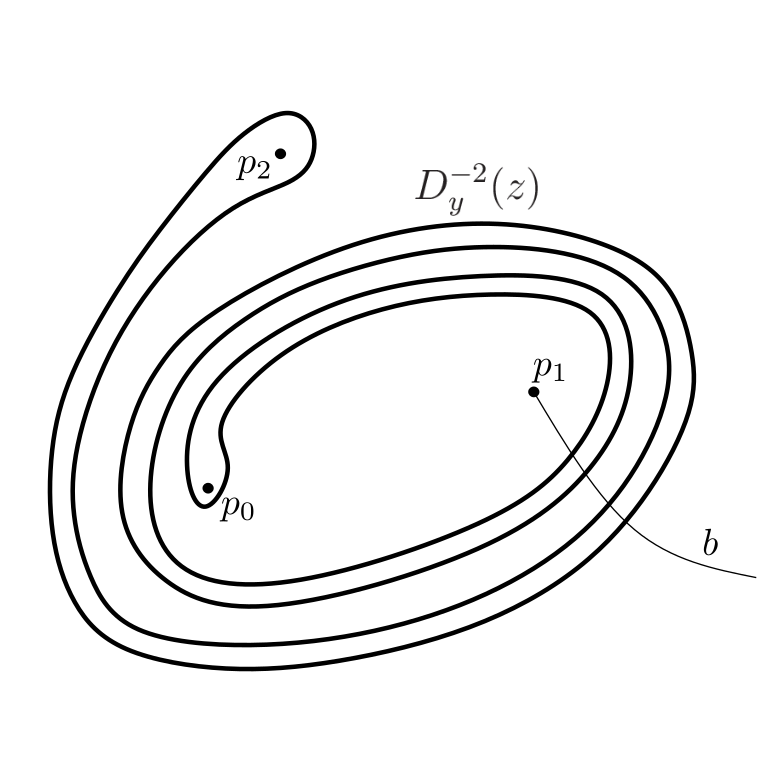}\quad
    \includegraphics[width=.4\textwidth]{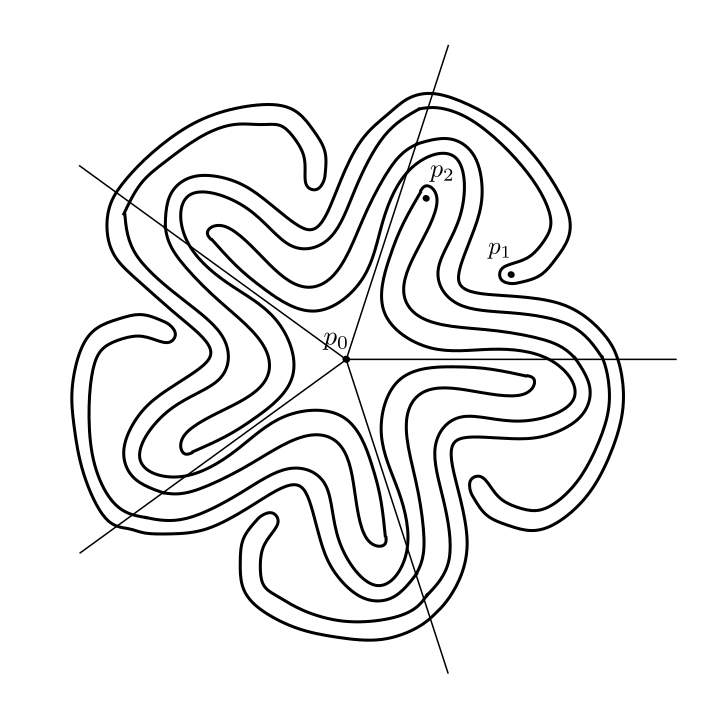}
    \caption{Left: The simple closed curves $y$ and $z$, Center: The simple closed curve $D_y^{-2}(z)$, Right: The preimage of $D_y^{-2}(z)$ under $R_d$ when $d=5$. Note that all components of this preimage are peripheral.}
    \label{liftedcurve}
\end{figure}
\item If $n=0,\ell\neq 0$, then we have:
$$D_x^m=D_x^{d^2k + d\ell} \rsquigarrow{}{3} D_y^{dk+\ell} \overset{D_y^\ell}{\rsquigarrow{}{5}} D_z^kD_y^\ell \overset{D_y^\ell}{\rsquigarrow{}{5}} \psi(D^k_{D_y^{-\ell}(z)})D_y^\ell = D_y^\ell.$$
    The last equality holds by Lemma~\ref{lem:branch_cut_intersection} since $D_y^{-\ell}(z)$ has algebraic intersection $\ell$ with the branch cut, and $d\nmid \ell$. 
    
     \item If $n\neq 0$, then we have:
    \begin{align*}
        D_x^{d^2k+d\ell+n}& \overset{D_x^n}{\rsquigarrow{}{5}} D_y^{dk+\ell}D_x^n = D_y^{dk+\ell-1}D_x^{-1}D_z^{-1}D_x^{n}\\ & 
        \overset{D_y^{dk+\ell-1}D_x^{n-1}}{\rsquigarrow{}{13}} \psi(D_{D_x^{-n}(z)}^{-1})D_y^{dk+\ell-1}D_x^{n-1}\\& = D_y^{dk+\ell-1}D_x^{n-1}. 
        \end{align*}
        The first equality follows from the lantern relation $D_xD_yD_z = id $.  The last line holds because $1\leq n\leq d-1$, so $d\nmid n$ and we may apply Lemma~\ref{lem:branch_cut_intersection}.
        
        The calculation above showed that $D_y^{dk+\ell}D_x^nR_d$ is equivalent to $D_y^{dk+\ell-1}D_x^{n-1}R_d.$  Performing this step $n$ times, we have that $D_y^{dk+\ell}D_x^nR_d$ is equivalent to  $$D_y^{dk+\ell-n}D_x^{n-n}R_d = D_y^{dk+\ell-n}R_d.$$ Finally, we have that
      \begin{align*}
          D_x^{d^2k+d\ell+n}R_d &\simeq D_y^{dk+\ell-n}R_d \\ & \simeq 
        \begin{cases}
        D_x^kR_d & \ell=n\\
        D_y^{\ell-n}R_d & \ell\neq n
        \end{cases}.
      \end{align*}
      \end{enumerate}
  \hspace{20pt}The last equivalence follows by cases (1) and (2). 
\end{proof}

\p{Example} We illustrate the reduction formulae where $d=5$ and  $m=23425$:
\begin{align*}
    23425 &= 5^2(937) + 5(0)+0\\
    937 & = 5^2(37) + 5(2) + 2\\
    37 & = 5^2(1) + 5(2) + 2\\
    1 & = 5^2(0) + 5(0) + 1.
 \end{align*} Therefore $D_x^mR_d \simeq D_x^{937}R_d \simeq D_x^{37}R_d \simeq D_xR_d \simeq \boxed{D_y^{-1}R_d}$.

\p{Reduction to base cases} More generally, for $m=d^2k+d\ell+n$ with $\ell=n$, Lemma~\ref{lem:reduction} returns a branched cover $D_x^kR_d$ with $k\leq m$.  We can then repeatedly apply Lemma~\ref{lem:reduction} until one of three ``base cases" occurs (we describe the process precisely in the proof of Theorem~\ref{thm:degree_d}):
\begin{enumerate}
    \item we obtain a map $D_y^{\ell-n}R_d$ for $0\leq|\ell-n|\leq d-1$ to which $D_x^mR_d$ is equivalent, 
    \item $m\geq 0$ and $D_x^mR_d$ reduces to $D_x^0R_d$, or
    \item $m<0$ and $D_x^mR_d$ reduces to $D_x^{-1}R_d$.
\end{enumerate}
It therefore suffices to compute $D_x^{-1}R_d$, and $D_y^iR_d$ for $1\leq i\leq d-1$ and $-(d-1)\leq i\leq -1$.

\section{Invariant trees}\label{sec:base_cases}
In this section, we will determine the base case polynomials $D_y^{i}R_d$ for $1\leq i\leq d-1$ and $-(d-1) \leq i \leq -1$, and for $D_x^{-1}R_d$ by finding their Hubbard trees. 

For each base case, we find an invariant trees using the tree lifting algorithm of Belk--Lanier--Margalit and the second author \cite{BLMW}.  However, we will not show this process.  We need only verify that each tree that we find is invariant under lifting and has an invariant angle structure under the corresponding map, thus satisfying the conditions of Poirier.  This proves that the invariant tree we found is indeed a Hubbard tree for the corresponding map.

\begin{figure}
    \centering
    \includegraphics[scale=.23]{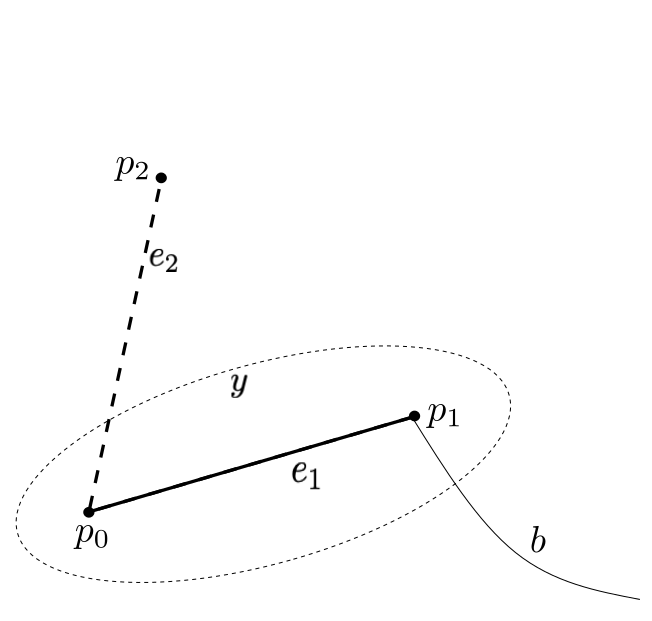}\quad \quad     \includegraphics[scale=.23]{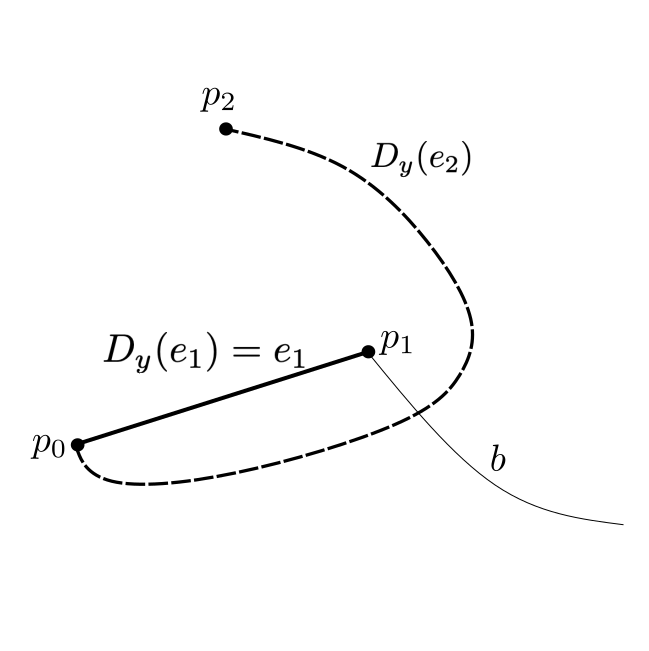}\quad\quad \includegraphics[scale=.23]{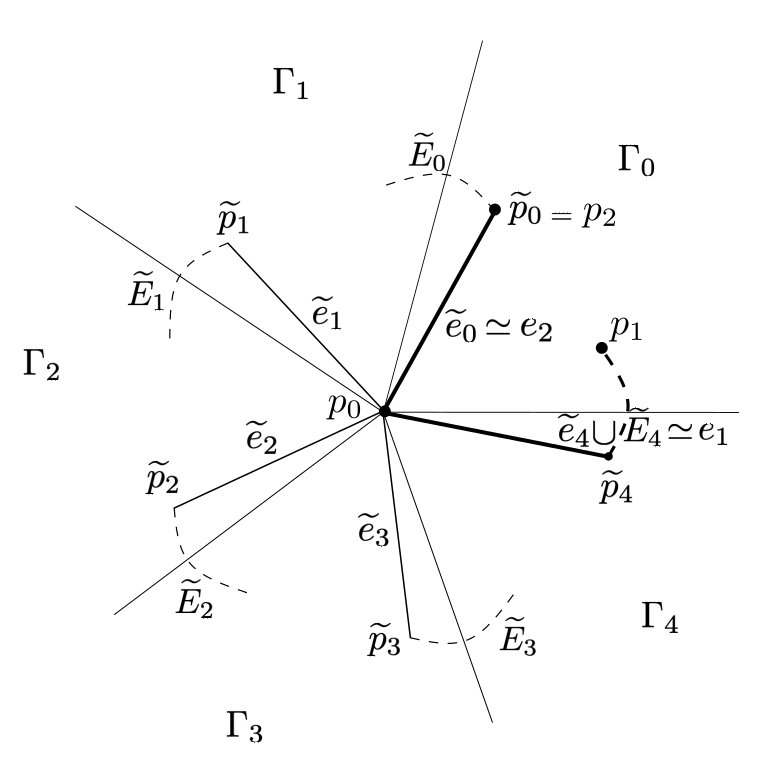}
    \caption{Left: The Hubbard tree $H_-$ for $D^i_yR_d$ with $-(d-1)\leq i\leq -1$. Center: a tree that is homotopic to $D_y(H_-)$.  Right: a tree that is homotopic to the preimage of $D_y(H_-)$ under $R_5$.}
    \label{fig:hubbard_tree_neg}
\end{figure}
\begin{proposition}\label{prop:base_case} Let $R_d$ be the degree-$d$ rabbit polynomial.  Then:
\begin{align*}
    D_y^i R_d & \simeq \begin{cases}
    A_{d,d-i} & 1\leq i\leq d-1\\
    A_{d,-i} & -(d-1)\leq i \leq -1\\
    \end{cases}
\end{align*}
\end{proposition}

The Hubbard tree $H_-$ for $D_y^iR_d$ for $-(d-1)\leq i\leq -1$ is the tree in Figure \ref{fig:hubbard_tree_neg}.  The invariant angle between $e_1$ and $e_2$ (measured counterclockwise) has measure $ - \frac{2\pi i}{d} = \frac{2 \pi |i|}{d}$.  The Hubbard tree $H_+$ for $D_y^iR_d$ for $1\leq i\leq d-1$ is the tree in Figure \ref{fig:hubbard_tree_pos}.  The invariant angle between $e'_1$ and $e_2$ (measured counterclockwise) has measure $2\pi - \frac{2\pi  i}{d}$.  Figure~\ref{fig:hubbard_tree_neg} demonstrates that $H_-$ is invariant under $R_5^{-1}D_y$, and Figure~\ref{fig:hubbard_tree_pos} demonstrates that $H_+$ is invariant under $R_5^{-1}D_y^{-1}$; the proof below explains that the figures generalize for all $d$ and for $-(d-1)\leq i\leq -1$ and $1\leq i\leq d-1$, respectively.

\begin{proof}
Let $b$ be a special branch cut for $R_d$ (for instance, the arc $b$ in Figure~\ref{fig:hubbard_tree_neg}). By the definition of special branch cut, the preimage $R_d^{-1}(b)$ consists of $d$ arcs from $p_0$ to $\infty$ such that the complement of $R_d^{-1}(b)$ in $\C$ contains $d$ components and one of them contains both points $p_1$ and $p_2$.  Label these complementary components by $\Gamma_0,\cdots, \Gamma_{d-1}$ counterclockwise where $\Gamma_0$ is the component that contains $p_1$ and $p_2$.  The set $R_d^{-1}(p_0)$ contains $d$ points; name them $\widetilde{p}_0,\cdots,\widetilde{p}_{d-1}$ such that $\widetilde{p}_j\in\Gamma_j$.

We will show that $H_-$ is the Hubbard tree for $D_y^iR_d$ when $-(d-1)\leq i\leq 1$ and $H_+$ is the Hubbard tree for $D_y^iR_d$ when $1\leq i\leq d-1$  using the same basic strategy.  First we take the lift of the tree $D_{y}^{-i}(H_\pm)$ under $R_d$ by computing the path lifts of the edges of $H_\pm$ (the edges of $H_-$ are $e_1$ and $e_2$ in $H_-$, the edges of $H_+$ are $e'_1$ and $e_2$) and determining which are in the hull of $p_0,p_1,$ and $p_2$.  We then verify the desired angle assignment is invariant under lifting.

We first treat the case where $-(d-1)\leq i\leq -1$.  Note that $i$ is negative so $D_y^{-i}$ is a positive (left-handed) twist.  Observe that $e_1$ is invariant under $D_y^{-i}$ for all $i$.  The edge $D_y^{-i}(e_2)$ intersects the branch cut $b$ with both algebraic and geometric intersection $i$ (with appropriately chosen orientations of $b$ and $e_2$).  That is: all intersections of $D_y^{-i}(e_2)$ and $b$ have the same orientation (this an arc version of \cite[Proposition 3.2]{primer}). 

For each $0\leq j\leq d-1$, there is a path lift of $D_y^{-i}(e_1)=e_1$ based at $\widetilde{p}_j$, call it $\widetilde{e}_j$. 
Because the interior of $e_1$ is disjoint from the branch cut $b$, each $\widetilde{e}_j$ is a straight line segment from $\widetilde{p}_j$ to $p_0=R_d^{-1}(p_1)$ that is contained in $\Gamma_j$, as in Lemma \ref{lem:branch_cut_intersection}.  In particular, $\widetilde{e}_0$ is contained in the closure of $\Gamma_0$ and has endpoints $p_0$ and $p_2$. Since $\widetilde{e}_0$ is homotopic to a straight line segment relative to $\{p_0, p_1, p_2\}$, it is homotopic to $e_2$.

The arc $D_y^{-i}(e_2)$ twists counterclockwise $|i|$ times around $e_1$ (when oriented from $p_0$ to $p_2$).   
For each $0\leq j\leq d-1$, there is a path lift of $D_y^{-i}(e_2)$ based at $\widetilde{p}_j$ with opposite endpoint in $R_d^{-1}(p_2)$; call this $\widetilde{E}_j$.  Each $\widetilde{E}_j$ comprises a distinct component of $R_d^{-1} (D_y^{-i}(e_2))$. Moreover, each $\widetilde{E}_j$ intersects $R_d^{-1}(b)$ at $|i|$ points, rotating counterclockwise from $\widetilde{p}_j$ to a point in $R^{-1}(p_2)$.  We may then apply Lemma~\ref{lem:branch_cut_intersection} to determine that the opposite endpoint of $\widetilde{E}_j$ is in $\Gamma_{(j+|i|)\mod d}$.

The preimage $R_d^{-1}(D_y^{-i}(H_-))$ is comprised of the union of $\{\widetilde{e}_j\mid 0\leq j\leq d-1\}$ and $\{\widetilde{E}_j\mid 0\leq j\leq d-1\}$.  For each $j$, the union of $\widetilde{e}_j$ and $\widetilde{E}_j$ is a path between an element of $R_d^{-1}(p_2)$ and $p_0$ (via $\widetilde{p}_j$).  The only such path that is in the hull of $\{p_0,p_1,p_2\}$ in $R_d^{-1}(D_y^{-i}(H_-))$ is  $\widetilde{e}_{d+i}\cup\widetilde{E}_{d+i}$, which contains $p_1$, $\widetilde{p}_{d+i}$, and $p_0$.  This edge is homotopic to $e_1$.  The other edge in the hull of $\{p_0,p_1,p_2\}$ in $R_d^{-1}(D_y^{-i}(H_-))$ is $\widetilde{e}_0$, which has endpoints $p_2$ and $p_0$ and is homotopic to $e_2$.  Thus the tree $H_-$ is invariant under lifting.

\begin{figure}
    \centering
    \includegraphics[scale=.3]{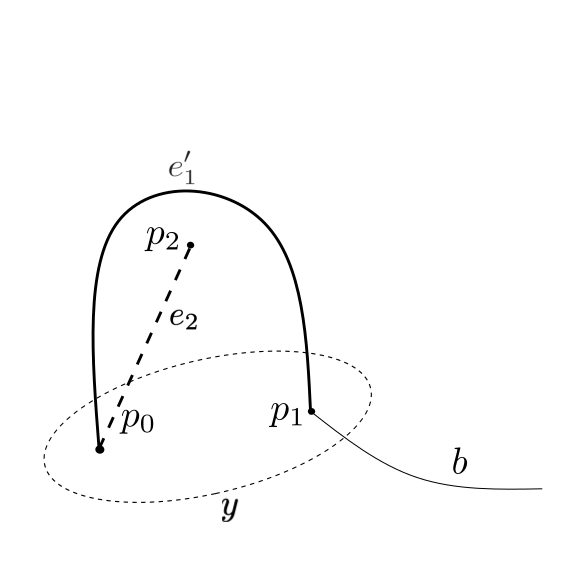}\quad \quad     \includegraphics[scale=.3]{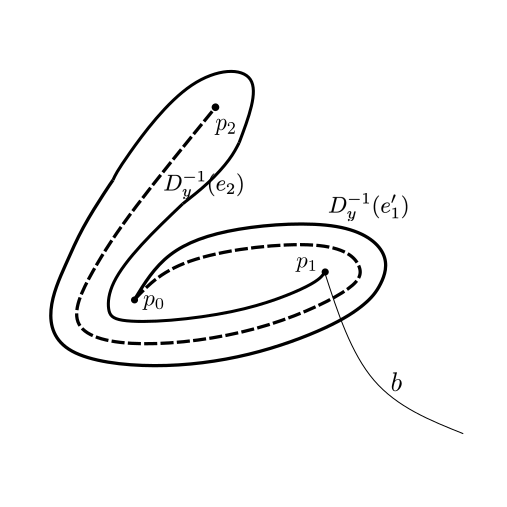}\quad\quad \includegraphics[scale=.25]{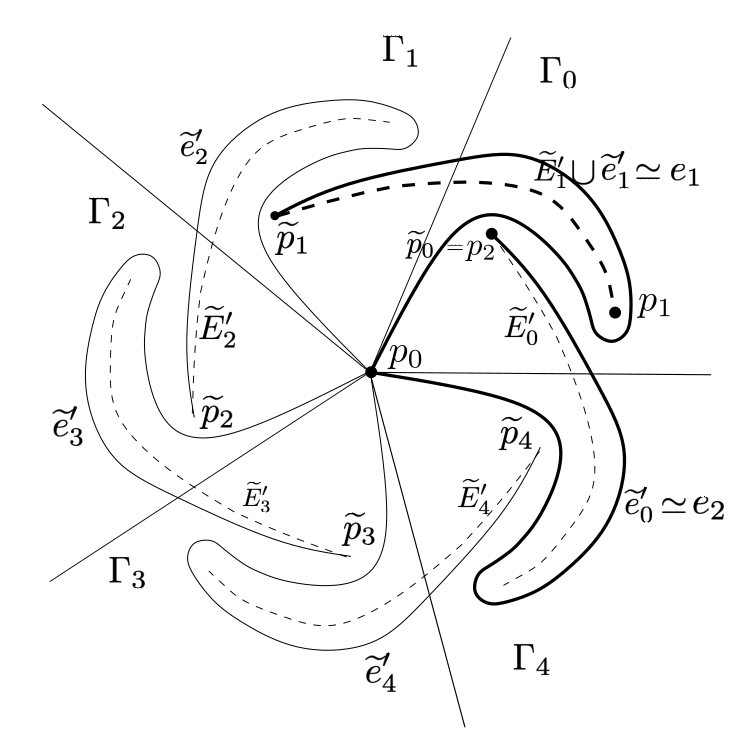}
    \caption{Left: The Hubbard tree $H_+$ for $D^i_yR_d$ with $1\leq i\leq d-1$. Center: a tree that is homotopic to $D_y^{-1}(H_+)$.  Right: a tree that is homotopic to the preimage of $D_y^{-1}(H_+)$ under $R_5$.}
    \label{fig:hubbard_tree_pos}
\end{figure}

To see that the angle between $e_1$ and $e_2$ (measured counterclockwise) of $-\frac{2\pi i}{d}$ is invariant under $D_y^iR_d$ for $-(d-1)\leq i\leq -1$, we track the preimage of all angles under the lifting process.  In particular, we observe that the preimage of the angle between $e_1$ and $e_2$ is an angle at an unmarked vertex (of valence 2) in $R_d^{-1}(D_y^{-i}(H_-))$, which is therefore irrelevant.  However, because $R_d^{-1}(p_1)=p_0$, the preimage of the angle of $2\pi$ at $p_1$ (measured from $e_1$ to itself) consists of $d$ angles between the $d$ path lifts of $e_1$, each of measure $\frac{2\pi}{d}$.  The path lifts of $e_1$  in the hull of $R_d^{-1}(D_y^{-i}(H_-))$ relative to $\{0,p_1,p_2\}$ are $\widetilde{e}_0$ and $\widetilde{e}_{d+i}$. There are $|i-1|$ path lifts of $e_1$ under $R_d$ between  $\widetilde{e}_{d+i}$ $\widetilde{e}_0$ in the cyclic counterclockwise ordering of vertices adjacent to $p_0$.  Therefore the angle between $\widetilde{e}_0$ and $\widetilde{e}_{d+i}$ is $\frac{2\pi |i|}{d}$, when measured counterclockwise.  
Since $i$ is negative, the counterclockwise angle between the edges of the lift homotopic to $e_1$ and $e_2$ respectively is $-\frac{2\pi i}{d}$, and we have shown that $D_y^iR_d$ is equivalent to $A_{d,-i}$ for $-(d-1)\leq i\leq -1$.

Now we show that the tree $H_+$ in Figure \ref{fig:hubbard_tree_pos} is invariant under $D_y^iR_d$ for $1\leq i\leq d-1$.  
Both edges $e'_1$ and $e_2$ (in Figure~\ref{fig:hubbard_tree_pos}) meet at the critical point $p_0$.  Orient both edges away from $p_0$.  
The arcs $D_y^{-i}(e'_1)$ and $D_y^{-i}(e_2)$ each intersect $b$ in $i$ points and all intersections have the same orientation (ie. both arcs are directed clockwise at the points of intersection with $b$).  For each $0\leq j\leq d-1$, there is a path lift of $D_y^{-i}(e_1)$ under $R_d$ based at $\widetilde{p}_j$ and a path lift of $D_y^{-i}(e_2)$ based at $\widetilde{p}_j$; call these $\widetilde{e}'_j$ and $\widetilde{E}'_j$ respectively.  Each $\widetilde{e}'_j$ and $\widetilde{E}'_j$ intersects $R_d^{-1}(b)$ at $i$ points in a clockwise direction until it reaches its other endpoint.  The other endpoint of $\widetilde{E}'_j$ is an element of $R_d^{-1}(p_2)$ and by Lemma~\ref{lem:branch_cut_intersection}, the this endpoint is in $\Gamma_{(j-i)\mod d}$. The other endpoint of $\widetilde{e}'_j$ is $p_0$ for all $j$.  
Then $\widetilde{e}'_0$, which has endpoints $p_2$ and $p_0$, is in the hull of $\{p_0,p_1,p_2\}$ in $R_d^{-1}(D_y^{-i}(H_+))$.  In fact, as long as $1\leq i\leq d-1$, $\widetilde{e}'_0$ is homotopic to a straight line segment, that is: $e_2$.  Moreover, the union of $\widetilde{E}'_i$ and $\widetilde{e}'_i$ forms  an edge from $p_1$ to $p_0$ (via $\widetilde{p}_i$). As long as $1\leq i\leq d-1$, this edge is homotopic to $e'_1$. 

As in the case where $-(d-1)\leq i\leq -1$, to verify that the angle $2\pi - \frac{2\pi i}{d}$ between $e'_1$ and $e_2$ (measured counterclockwise) is invariant under $R_d^{-1}D_y^{-i}$, we need only consider the angle between the path lifts of $e'_1$ that are in the hull of $\{p_0,p_1,p_2\}$.  Indeed, we saw above that $\widetilde{e}'_0$ is always homotopic to $e_2$ and $\widetilde{e}'_i$ is part of the path in $R_d^{-1}(D_y^{-i}(H_+))$ that is homotopic to $e'_1$.  The angle between $\widetilde{e}'_i$ and $\widetilde{e}'_0$ (measured counterclockwise) then has measure $2\pi-\frac{2\pi i}{d}$, as desired.  Therefore $H_+$ with angle measure $2\pi - \frac{2\pi i}{d}$ between $e'_1$ and $e_2$ is invariant under $R_d^{-1}D_y^{-i}$ when $1\leq i\leq d-1$. Thus $D_y^iR_d$ is equivalent to $A_{d,d-i}$ for $1\leq i\leq d-1$.
\end{proof}
\begin{figure}
    \centering
    \includegraphics[width=.25\textwidth]{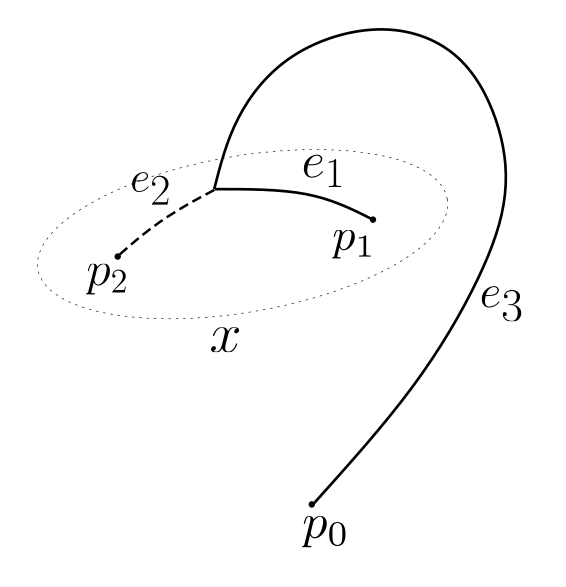}\quad \quad
    \includegraphics[width=.25\textwidth]{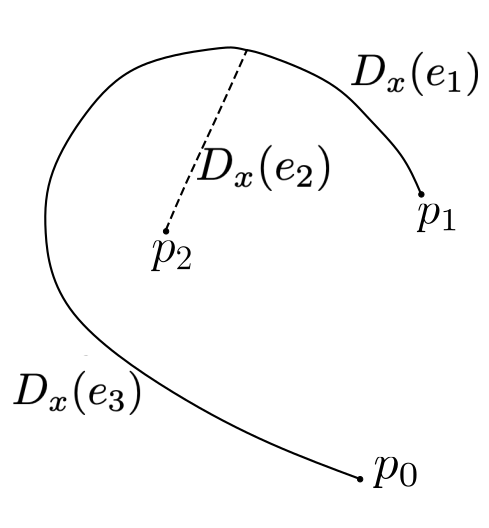}\quad \quad
    \includegraphics[width=.25\textwidth]{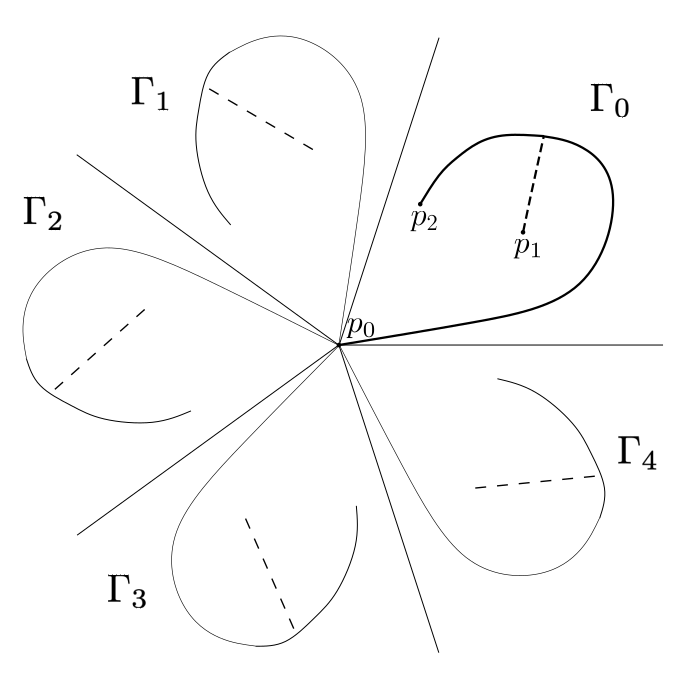}
    \caption{Left: The Hubbard tree $H$ for $D_x^{-1}R_d$, Center: $D_x(H)$, Right: $R_5^{-1}(D_x(H))$.}
    \label{fig:corabbit}
\end{figure}

\begin{lemma}\label{lem:corabbit}
For all $d\geq 2$, the branched cover $D_x^{-1}R_d$ is equivalent to $\overline{R}_d$.
\end{lemma}

\begin{proof}
Figure~\ref{fig:corabbit} shows the Hubbard tree $H$ for $D_x^{-1}R_d$ for all $d\geq 2$.  Figure~\ref{fig:corabbit} also shows that the lift of $D_x(H)$ under $R_5$ is homotopic to $H$; a similar calculation verifies that the lift of $D_x(H)$ under $R_d$ is homotopic to $H$ for all $d\geq 2$.  The angle assignment of $2\pi/3$ at each angle of the trivalent vertex is invariant under lifting.  There are no Julia edges, so Poirier's conditions verify that $H$ is indeed the Hubbard tree of $D_x^{-1}R_d$.  Moreover, $D_x^{-1}R_d$ rotates the edges clockwise relative to the trivalent vertex (one edge and its lift is dashed to assist tracking the rotation).  Therefore $D_x^{-1}R_d$ is equivalent to $\overline{R}_d$ for $d\geq 2$.
\end{proof}

\section{Proof of Main Theorem}
We now combine the reduction formulae from Section~\ref{sec:reduction} and the base cases from Section~\ref{sec:base_cases} to prove Theorem~\ref{thm:degree_d}.
\begin{proof}[Proof of Theorem \ref{thm:degree_d}]
  As in the theorem, we consider the map $D_x^mR_d$.  Consider the $d^2$-adic expansion of $m$, which is $m_sm_{s-1}\cdots m_1$ if $m\geq 0$ and $\overline{d^2-1}m_sm_{s-1}\cdots m_1$ if $m<0$.  

We may write $m $ uniquely as 
\begin{align*}
    m=d^2k_1+ d\ell_1 +n_1
\end{align*}
where $k_1\in \Z$, $\ell_1,n_1 \in \{0,1,...,d-1\}$. We note that 
    $$k_1=\frac{m-m_1}{d^2}\text{ and }
    m_1=d\ell_1+n_1.$$
For $2\leq i \leq s$, there exist integers $k_i \in \Z$, $0 \leq \ell_i, n_i \leq d-1$ such that 
\begin{align*}
     k_{i} &= \frac{k_{i-1}-m_i}{d^2},\\
     m_i & = d\ell_i + n_i,\text{ and }\\
     k_{i-1} &= d^2k_i + d\ell_i + n_i.
\end{align*}

That is: the $d^2$-adic expansion of $k_i$ is obtained from $m$ by dropping the last (right-most) $i$ digits.  In particular, if $m\geq 0$, then $k_s=0$.  If $m<0$, then $k_s$ has $d^2$-adic expansion $\overline{d^2-1}$ and therefore $k_s=-1$.
Furthermore, we note that $\ell_i=n_i$ if and only if $(d+1)|m_i$.

By applying Lemma~\ref{lem:reduction} to $D_x^mR_d$ we have
\begin{align*}
    D_x^mR_d \simeq \begin{cases}
    D_x^{k_1}R_d &\text{if } (d+1)|m_1\\
    D_y^{\ell_1-n_1}R_d & \text{otherwise}
    \end{cases}.
\end{align*}
Thus if $(d+1)\nmid m_1$, we may apply  Proposition~\ref{prop:base_case}, to obtain: \begin{align*}
   D_x^mR_d\simeq  D_y^{\ell_1-n_1}R_d & = \begin{cases}
        A_{d,n_1-\ell_1} &\text{if } n_1>\ell_1\\
        A_{d,d-(\ell_1-n_1)} &\text{if } n_1<\ell_1
    \end{cases}.
\end{align*}
So we can deduce the equivalence class of $D_x^mR_d$ directly if $\ell_1 \neq n_1$.  Otherwise, we right-shift the $d^2$-adic expansion of $m$, and consider $D_x^{k_1}R_d$ instead. We repeat this process for $k_1,\cdots,k_s$.  Then one of the following will occur:
\begin{itemize}
    \item If $(d+1)|m_i$ for all $i$, then $\ell_i=n_i$ for all $i$.  In particular $\ell_s=n_s$ and we have
    \begin{align*}
    D_x^mR_d\simeq D_x^{k_1}R_d\simeq D_x^{k_2}R_d\simeq\cdots \simeq D_x^{k_{s-1}}R_d\simeq D_x^{k_s}R_d. 
\end{align*}

If $m\geq 0$, then $k_s=0$, so $D_x^mR_d\simeq D_x^{k_s}R_d\simeq R_d$.

If $m<0$, then $k_s=\overline{d^2-1}=-1$, and $D_x^mR_d\simeq D_x^{k_s}R_d\simeq D_x^{-1}R_d$.  By Lemma~\ref{lem:corabbit}, $D_x^{-1}R_d$ is equivalent to $\overline{R}_d$.
\item If there exists $i$ such that $d+1$ does not divide $m_i$, choose the minimal such $i$ and write $m_i = d\ell_i+n_i$, with $\ell_i, n_i \in \{0,1,...,d-1\}$.  Then $\ell_i \neq n_i$.  Thus by Proposition~\ref{prop:base_case}, we have:
\begin{align*}
    D_x^mR_d \simeq  D_y^{\ell_i-n_i}R_d \simeq \begin{cases}
    A_{d,n_i-\ell_i}&\text{if } n_i>\ell_i\\
    A_{d,d-(\ell_i-n_i)} &\text{if } n_i<\ell_i
    \end{cases}.
\end{align*}
\end{itemize}
This completes the proof of the theorem.
\end{proof}

\bibliographystyle{plain}
\bibliography{refs}
\end{document}